\documentclass[12pt]{article}
\usepackage[mathscr]{euscript}
\usepackage{amsmath,amssymb}
\usepackage{verbatim}
\usepackage{dsfont,bm}
\usepackage{color}
\usepackage{graphicx}
\usepackage{amsthm}
\usepackage{enumerate}
\usepackage{esint} 

\usepackage{xcolor}
\definecolor{darkgreen}{rgb}{0,0.75,0}
\definecolor{darkred}{rgb}{0.75,0,0}
\definecolor{darkmagenta}{rgb}{0.5,0,0.5}
\usepackage[colorlinks,citecolor=darkgreen,linkcolor=darkred,urlcolor=darkmagenta]{hyperref}
\newtheorem{theorem}{Theorem}[section]

\newtheorem{cor}[theorem]{Corollary}

\newtheorem{lem}[theorem]{Lemma}

\newtheorem{prop}[theorem]{Proposition}

\theoremstyle{definition}
\newtheorem{definition}[theorem]{Definition}

\newtheorem{remark}[theorem]{Remark}

\numberwithin{equation}{section}

\def\be{\begin{equation}}
\def\ee{\end{equation}}
\def\bes{\begin{equation*}}
\def\ees{\end{equation*}}

\newcommand{\set}[1]{\left\{ #1 \right\}}

\newcommand{\abs}[1]{{\left\vert\kern-0.25ex #1
		\kern-0.25ex\right\vert}}
\newcommand{\one}{\mathds{1}} 

\newcommand{\diag}[0]{\operatorname{diag}}
\DeclareMathOperator*{\esssup}{ess\,sup}
\DeclareMathOperator*{\essinf}{ess\,inf}

 \def\sE {{\mathcal E}} \def\sF {{\mathcal F}}
  
  \def\sL {{\mathcal L}}

 \def\bE {{\mathbb E}}

\def\bP {{\mathbb P}}  \def\bR {{\mathbb R}}

\def\sms{\smallskip}

\def\ignore#1{}

\def\lam {\lambda}  
\def\eps{\varepsilon}

\def\Gam{\Gamma}

\def\to {\rightarrow}

\def\q{\quad} \def\qq{\qquad}
\def\dint{\int\kern-.6em\int}

\newcommand\restr[2]{{
		\left.\kern-\nulldelimiterspace 
		#1 
		\vphantom{\big|} 
		\right|_{#2} 
	}} 
	
	\def\supp{\mathop{{\rm supp}}}

	\newcommand{\on}[1]{\operatorname{ #1}}
	
	\def\wt{\widetilde}
	\def\wh{\widehat}
	
	\def\be{\begin{equation}}
	\def\ee{\end{equation}}
	\def\bes{\begin{equation*}}
	\def\ees{\end{equation*}}
	\def\ba{\begin{align}}
	\def\ea{\end{align}}
	\def\xxea{\end{align}}
\def\bas{\begin{align*}}
\def\eas{\end{align*}}

\def\proof{{\smallskip\noindent {\em Proof. }}}
\def\qed{{\hfill $\square$ \bigskip}}

\definecolor{dgreen}{rgb}{0, 0.6, 0.1}
\definecolor{dblue}{rgb}{0, 0.0, 0.6}
\definecolor{vdblue}{rgb}{0,.08, 0.45}
\definecolor{dred}{rgb}{0.7, 0.0, 0.0}
\definecolor{vdblue}{rgb}{0,.08, 0.45}

\definecolor{purple}{rgb}{0.6, 0.0, 0.6}
\definecolor{mytext}{rgb}{0.1, 0.1, 0.1}

\begin{document}
	
	\font\titlefont=cmbx14 scaled\magstep1
	\title{ On the comparison between jump processes and subordinated diffusions
		 \titlefont }    
	\author{	Guanhua Liu\footnote{Research partially supported by China Scholarship Council.} ,	Mathav Murugan\footnote{Research partially supported by NSERC (Canada) and the Canada research chair program.} 
		}
	\maketitle
	\vspace{-0.5cm}
\begin{abstract}
	Given a symmetric diffusion process and a jump process on the same underlying space, is there a subordinator such that the jump process and the subordinated diffusion processes are comparable? 
	We address this question when the diffusion satisfies a sub-Gaussian heat kernel estimate and the jump process satisfies a polynomial-type jump kernel bounds.
	Under these assumptions , we obtain necessary and sufficient conditions on the jump kernel estimate for such a subordinator to exist.
	As an application of our results and the recent stability results of Chen, Kumagai and Wang, we obtain parabolic Harnack inequality   for a large family of jump processes.
In particular, we show that any jump process  with polynomial-type jump kernel bounds on such a space satisfy the parabolic Harnack inequality.

\end{abstract}

\section{Introduction}
Let $(X(t))$ and $(Y_{\alpha}(t))$ denote the Brownian motion and symmetric $\alpha$-stable process on $\bR^n$ respectively, where $\alpha \in (0,2)$.
These process form the basic examples of symmetric diffusions and jump processes respectively.
 The jump kernel of the $\alpha$-stable process $(Y_{\alpha}(t))$ on $\bR^n$ is given by
\[
J(x,y) =  \frac{c_{n,\alpha}}{d(x,y)^{n+\alpha}}, \q \mbox{for all $x,y \in \bR^n$,}
\]
where $d$ denotes the Euclidean distance.
The processes  $(X(t))$ and $(Y_{\alpha}(t))$ are related via a  \emph{subordinator}. A subordinator is a one-dimensional L\'evy process with non-decreasing paths. The process $(Y_\alpha(t))$ has the same law as $(X(S(t)))$, where $(S(t))$ is a subordinator independent of $(X(t))$ and defined by its Laplace transform $\bE \exp(-\lam S(t))= \exp (-t \lam^{\alpha/2} )$ for all $t, \lam \ge 0$. Therefore one could study $\alpha$-stable processes using properties of Brownian motion and the subordinator. In particular, we have
\be \label{e:subjump}
\bP^x(Y_\alpha (t) \in A)= \int_0^\infty \bP^x (X(s) \in A) \, \eta_t(ds),
\ee
where $\eta_t$ is the law of the subordinator $S(t)$ described above.

A well-known important application of \eqref{e:subjump} is that one can obtain heat kernel bounds and parabolic Harnack inequality for the jump process $(Y_\alpha)$ by transferring heat kernel bounds and parabolic Harnack inequality for the diffusion $(X)$ along with heat kernel estimates on the subordinator (\cite[Section 5.2]{CKW3} and \cite[Section 4.1]{BKKL}).
In this work, we investigate the extent to which subordination can be used to jump processes. In particular, our work addresses the following questions:
\begin{enumerate}[(a)]
	\item Let $X$ be a $\mu$-symmetric diffusion that satisfies the parabolic Harnack inequality. Given a $\mu$-symmetric jump process $Y$, does there exist a subordinator $S$ such that the subordinated process $X \circ S$ has jump kernel comparable to that of $Y$? 
	\item In the setting above, does the jump process $Y$ also inherit the parabolic Harnack inequality from $X$? If so, what is the space time scaling of the process $Y$?
\end{enumerate}
Under fairly mild conditions on the jump kernel, we obtain a positive answer to both these questions.
Our answer to question (a) seems to be new even on $\bR^n$ (see Remark \ref{r:prev}). The motivation for question (a) arises from the beautiful recent  results of Chen, Kumagai, Wang concerning the stability of  parabolic Harnack inequality for jump processes \cite{CKW2}. Using their results a positive answer from question (b) follows from a positive answer to question (a).

\section{Framework and results}

\subsection{Dirichlet forms and symmetric Markov processes}
Throughout this work, we consider a complete, locally compact, separable, unbounded metric space $(M,d)$ equipped with a Random measure $\mu$ with full support, i.e., a Borel measure $\mu$ on $M$ that is finite on any compact set and positive on any non-empty open set. Such a triple $(M,d,\mu)$  is called a \emph{metric measure space}.

Let $(\mathcal{E},\mathcal{F})$ be a \emph{symmetric Dirichlet form} on $L^{2}(M,\mu)$.
That is, the domain $\mathcal{F}$ is a dense linear subspace of $L^{2}(M,\mu)$, such that
$\mathcal{E}:\mathcal{F}\times\mathcal{F}\to\mathbb{R}$
is a non-negative definite symmetric bilinear form which is \emph{closed}
($\mathcal{F}$ is a Hilbert space under the inner product $\sE_{1}(\cdot,\cdot):= \sE(\cdot,\cdot)+ \langle \cdot,\cdot \rangle_{L^{2}(M,\mu)}$)
and \emph{Markovian} (the unit contraction operates on $\sF$; $\wh u:=(u \vee 0)\wedge 1\in\mathcal{F}$ and $\mathcal{E}(\wh u, \wh u)\leq \mathcal{E}(u,u)$ for any $u\in\mathcal{F}$).
Recall that $(\mathcal{E},\mathcal{F})$ is called \emph{regular} if
$\mathcal{F}\cap\mathcal{C}_{\mathrm{c}}(M)$ is dense both in $(\mathcal{F},\mathcal{E}_{1})$
and in $(\mathcal{C}_{\mathrm{c}}(M),\|\cdot\|_{\mathrm{sup}})$.
Here $\mathcal{C}_{\mathrm{c}}(M)$ is the space of $\mathbb{R}$-valued continuous functions on $M$ with compact support.

Given a Dirichlet form $(\sE,\sF)$, there is an associated \emph{Markov semigroup} $(P_t)_{t \ge 0}$ on $L^2(M,\mu)$ and a non-positive definite self-adjoint \emph{generator} $\sL$ such that $P_t= e^{t \sL}$. Furthermore, by \cite[Theorem 1.3.1 and Lemma 1.3.4]{FOT} the Dirichlet form $(\sE,\sF)$ is given in terms of the semigroup by
\begin{align}
\sF&= \set{f \in L^2(M,\mu): \lim_{t \downarrow 0} \frac{1}{t}\langle f - P_t f, f \rangle < \infty }, \label{e:df1} \\
\sE(f,f) &=  \lim_{t \downarrow 0} \frac{1}{t}\langle f - P_t f, f \rangle, \q \mbox{for all $f \in \sF$,} \label{e:df2}
\end{align}
where $\langle \cdot,\cdot \rangle$ denotes the $L^2(M,\mu)$ inner product.
Recall that by the spectral representation, $t \mapsto \frac{1}{t}\langle f - P_t f, f \rangle$ is non-increasing and
\be \label{e:df3}
\limsup_{t \downarrow 0}\frac{1}{t}\langle f - P_t f, f \rangle = \lim_{t \downarrow 0} \frac{1}{t}\langle f - P_t f, f \rangle    \q \mbox {for any $f \in L^2(M,\mu)$.}
\ee
It is known that the semigroup extends to a contraction on any $L^p(M,\mu)$, where $p \in [1,\infty]$. The Markov semigroup is said to be \emph{conservative} if $P_t 1 = 1$ for any $t >0$.

Associated with a regular Dirichlet form $(\sE,\sF)$ on $L^2(M,\mu)$ is a $\mu$-symmetric \emph{Hunt process} $(X_t, t \ge 0, \mathbb{P}_x, x \in  M \setminus \mathcal{N})$, where $\mathcal N$ is a properly exceptional set for $(\sE,\sF)$.  
Recall that a Hunt process is a strong Markov process that is   right continuous and quasi-left continuous on the one-point compactification
$M_\partial:=M \cup\{\partial \}$ of $M$. 
The \emph{heat kernel} associated with the Markov semigroup $\{P_t\}$ (if it exists) is a family of measurable functions $p(t,\cdot,\cdot):M  \times M  \mapsto [0,\infty)$ for every $t > 0$, such that
\begin{align}
P_tf(x) &= \int p(t,x,y)  f(y)\, \mu(dy), \q \mbox{for all $f \in L^2(M,\mu), t>0$ and $x \in M$,} \label{e:hkint} \\
p(t,x,y) &= p(t,y,x), \q \mbox{for all $x, y \in M$ and $t>0$,} \label{e:hksymm} \\
p(t+s,x,y)&= \int p(s,x,y) p(t,y,z) \,\mu(dy), \q \mbox{for all $t,s>0$ and $x,y \in M$.} \label{e:ck}
\end{align}

We recall the notion of strongly local and pure jump type Dirichlet forms.
For a Borel measurable function $f: M \to \bR$ or an $\mu$-equivalence class $f$ of such functions, $\supp_{\mu}[f]$ denotes the support of the measure $\abs{f}\,d\mu$, i.e., the smallest closed subset $F$ of $M$ with $\int_{M \setminus F} \abs{f}\,d\mu=0$, which exists since $M$ is separable. 
A Dirichlet form $(\sE,\sF)$ on $L^2(M,\mu)$ is said to be \emph{strongly local} if $\sE(f,g)=0$ for all functions $f,g \in \sF$ with $\supp_\mu[f],\supp_\mu[g]$ compact and $\supp_{\mu}[f-a\one_{M}] \cap \supp_{\mu}[g] = \emptyset$ for some $a \in \bR$. We say that a Dirchlet form $(\sE,\sF)$ on $L^2(M,\mu)$ is of pure jump type, if  there exists a symmetric positive Radon measure $\wt J$ on $M \times M \setminus \diag$ such that 
\[ 
\sE(f,f) = \int_{M \times M \setminus \diag} (f(x)-f(y))^2\, \wt J(dx,dy), \q \mbox{for all $f \in \sF$,}
\]
where $\diag = \set{(x,x) \mid x \in M}$ denotes the diagonal. The Radon measure $\wt J$ is called the \emph{jumping measure}; we refer to the Beurling-Deny decomposition for the reason behind this terminology \cite[Theorem 3.2.1 and Lemma 4.5.4]{FOT}. A symmetric Borel measurable function $J: M \times M \setminus \diag \to [0,\infty)$ is said to be a \emph{jump kernel} of a  Dirichlet form $(\sE,\sF)$ on $L^2(M,\mu)$ of pure jump type, if $\wt J(dx, dy)= J(x,y) \mu (dx)\mu(dy)$, where $\wt{J}$ is the jumping measure.

 Let $\bR_+=[0,\infty)$. We say that a homeomorphism $\psi:\bR_+ \to \bR_+$ is a \emph{scale function} if there exist $C \ge 1,0< \beta_1 \le \beta_2$ such that
 \be \label{e:sf}
C^{-1} \left(\frac{R}{r} \right)^{\beta_1} \le \frac{\psi(R)}{\psi(r)} \le C \left(\frac{R}{r} \right)^{\beta_2}, \q \mbox{for all $0< r \le R$.}
 \ee
Let $\psi_j$ be a scale function and let $\wt J$ be a jumping measure. We say that the jumping measure $\wt J$ satisfies \hypertarget{jpsij}{$\on{J(\psi_j)}$} if there exists a density $J$ such that  $\wt J(dx,dy)=J(x,y)\mu(dx) \mu(dy)$ and there exists $C>0$ such that
\be \tag{$\on{J(\psi_j)}$} \label{jpsij}
\frac{C^{-1}}{\mu(B(x,d(x,y)))\psi_j(d(x,y))}  \le J(x,y) \le \frac{C}{\mu(B(x,d(x,y)))\psi_j(d(x,y))},
\ee
for $\mu$-a.e.~$x,y \in M \times M \setminus \on{diag}$.
If the density $J$ satisfies the upper or lower bounds on $J$ in the above estimate, we say that the jumping measure satisfies 
\hypertarget{jpsijle}{$\on{J(\psi_j)_\le}$} or \hypertarget{jpsijge}{$\on{J(\psi_j)_\ge}$} respectively.
Jump process satisfying \ref{jpsij} have been widely studied in the context of heat kernel estimates and Harnack inequalities \cite{CKW3,BKKL, BKKL1,MS}.

\subsection{Parabolic Harnack inequality} 

We recall the definition of parabolic Harnack inequality and it's relationship to heat kernel bounds. Let 
$(\sE,\sF)$ be a Dirichlet form on $L^2(M,\mu)$ and let $I$ be an open interval in $\bR$. We say that a function $u: I \to L^2(M,\mu)$ is weakly differentiable at $t_0 \in I$ if the function $t \mapsto \langle u(t), f \rangle$ is differentiable at $t_0$ for all $f \in L^2(M,\mu)$ , where $\langle \cdot, \cdot \rangle$ denotes the inner product in $L^2(M,\mu)$. By the   uniform boundedness principle, there exists a (unique) function $w \in L^2(M,\mu)$ such that
\[
\lim_{t \to t_0} \left\langle \frac{u(t)-u(t_0)}{t- t_0}, f \right\rangle = \langle w , f \rangle, \q \mbox{ for all $f \in L^2(M,\mu)$.}
\]
In this case, we say that the function $w$ above is the \emph{weak derivative} of   $u$ at $t_0$ and write $w=u'(t_0)$. 
 Let $\Omega$ be an open subset of $M$.
A function $u: I \to \sF$ is said to be \emph{caloric} in $I \times \Omega$ if  $u$ is weakly differentiable in the space $L^2(\Omega)$ at any $t \in I$, and for any $f \in \sF \cap C_{\mathrm{c}}(\Omega)$, and for any $t \in I$,
\be
\langle u', f \rangle + \sE(u,f)=0. \label{e:caloric}
\ee

\begin{definition}[Parabolic Harnack inequality] \label{d:harnack}
Let $(\sE,\sF)$ be a Dirichlet form on $L^2(M,\mu)$ and let $\psi$ be a scale function. 
We say that  an MMD space $(X,d,m,\sE,\sF)$ satisfies the \emph{parabolic Harnack inequality} with walk dimension $\beta$ (abbreviated as \hypertarget{phi}{$\on{PHI(\psi)}$}),
if there exist $0<C_1< C_2 < C_3 < C_4 <\infty$, $C_5>1$ and $\delta \in (0,1)$ such that for all $x \in X$, $r>0$ and for any non-negative bounded caloric function $u$ on the space-time cylinder $Q=(a,a+\psi(C_4 r)) \times B(x,r)$, we have
\begin{equation} \label{PHI} \tag*{$\on{PHI(\psi)}$}
\esssup_{Q_-} u \le C_5 \essinf_{Q_+} u,
\end{equation}
where $Q_-=(a+  \psi(C_1r),a+ \psi(C_2r)) \times B(x,\delta r)$ and $Q_+=(a+  \psi(C_3r),a+ \psi(C_4r) s) \times B(x,\delta r)$.
\end{definition}

We recall the following sub-Gaussian heat kernel estimate  that is equivalent to the above parabolic Harnack inequality.

Let $(\sE,\sF)$ be a strongly local, regular, Dirichlet form on $L^2(M,\mu)$ and let $\psi$ be a scale function.  We say that the Dirichlet form $(\sE,\sF)$  on  $L^2(M,\mu)$ satisfies
\hypertarget{hke}{$\on{HKE}(\psi)$}, if there exist $C_{1},c_{1},c_{2},c_{3},\delta\in(0,\infty)$
and a heat kernel $\set{p_t}_{t>0}$ such that for any $t>0$,
\begin{align}\label{e:uhke}
p_{t}(x,y) &\leq \frac{C_{1}}{m\bigl(B(x,\psi^{-1}(t))\bigr)} \exp \bigl( -c_{1} \Phi( c_{2}d(x,y), t ) \bigr)
\qquad \mbox{for $\mu$-a.e.\ $x,y \in M$,}\\
p_{t}(x,y) &\ge \frac{c_{3}}{m\bigl(B(x,\psi^{-1}(t))\bigr)}
\qquad \mbox{for $\mu$-a.e.\ $x,y\in M$ with $d(x,y) \leq \delta\psi^{-1}(t)$,}
\label{e:nlhke}
\end{align}
where
\begin{equation} \label{e:defPhi}
\Phi(R,t) := \Phi_{\psi}(R,t) := \sup_{r>0} \biggl(\frac{R}{r}-\frac{t}{\psi(r)}\biggr),
\qquad \mbox{for all $R \ge 0, t >0$.}
\end{equation}
We recall volume doubling and reverse volume doubling properties of a metric measure space.
We say that a metric measure space $(M,d,\mu)$ satisfies the \emph{volume doubling property} \hypertarget{vd}{$\operatorname{VD}$} if there exists $C_D>1$ such that\[  \tag*{$\on{VD}$}
\mu(B(x,2r))
 \le \mu(B(x,r)), \qq \mbox{for all $x \in M, r>0$.}\]
 We say that a metric measure space $(M,d,\mu)$  satisfies the \emph{reverse volume doubling property} \hypertarget{rvd}{$\operatorname{RVD}$}, if there exists $A,C>1$ such that
 \[
 \mu(B(x,Ar))
 \ge C \mu(B(x,r)),\qq \mbox{for all $x \in M, r>0$.}
 \]
 We recall the following well-known equivalence between parabolic Harnack inequality and heat kernel estimates. 
 \begin{theorem}\cite[Theorem 3.1]{BGK} \label{t:phi-hke}
 	Let $(M,d,\mu)$ be a metric measure space that satisfies the \hyperlink{vd}{$\operatorname{VD}$}, \hyperlink{rvd}{$\operatorname{RVD}$}  and let $\psi$ be a scale function. Then the parabolic Harnack inequality  \hyperlink{phi}{$\on{PHI(\psi)}$} is equivalent to the heat kernel estimate \hyperlink{hke}{$\on{HKE}(\psi)$}
 \end{theorem}
\proof
The implication  \hyperlink{hke}{$\on{HKE}(\psi)$} implies \hyperlink{phi}{$\on{PHI(\psi)}$} follows from \cite[Theorem 3.1]{BGK} along with \cite[Lemma 3.19]{GT12}. The converse implication follows from \cite[Theorem 3.1]{BGK}, \cite[Theorem 4.2]{GT12} along with \cite[Theorem 1.2]{GHL15}. We  refer the reader to \cite[Theorem 4.5 and Remark 4.6]{KM} for further discussion on related results. \qed

\subsection{Subordinator and L\'evy measure}
A subordinator $(S_t)$ is a non-decreasing L\'evy process with $S_0=0$; that is, $(S_t)$ has independent, stationary increments such that $t \mapsto S_t$ is continuous in probability. A subordinator is characterized by its \emph{Laplace exponent} $\phi: (0,\infty) \to [0,\infty)$ such that
\be \label{e:laplace-exp}
\bE e^{-\lam S_t} = e^{-t \phi(\lam)},
\ee
where $\phi$ is a \emph{Bernstein function} determined by its \emph{drift} $a \in [0,\infty)$ and \emph{L\'evy measure} of the subordinator $\nu$, where $\nu$ is a Borel measure on $(0,\infty)$ such that 
\be \label{e:levymeas}
\int_{(0,\infty)} (1 \wedge s)\, \nu(ds) <\infty, \quad \mbox{and} \quad \phi(\lambda)= a \lambda +\int_0^\infty (1-e^{-s \lam}) \,\nu(ds).
\ee
Conversely, any drift $a \in [0,\infty)$ and Borel measure $\nu$ on $(0,\infty)$ that satisfy \eqref{e:levymeas} uniquely determine the subordinator $(S_t)$.
We refer the reader to \cite[Chapter 5]{SSV} or \cite[Chapter 6]{Sat} for background on subordinators.\\
{\bf Notation}. In the following, we will use the notation  $A \lesssim B$ for quantities $A$ and $B$ to indicate the existence of an
implicit constant $C \ge 1$ depending on some inessential parameters such that $A \le CB$. We write $A \asymp B$, if $A \lesssim B$ and $B \lesssim A$.

\subsection{Statement of the main results}
We now state the main results of this work.
The following theorem characterizes all polyomial type jump kernels on a metric measure space that admits a diffusion satsfying parabolic Harnack inequality. Our theorem establishes a one-to-one correspondence between  jump processes with polynomial type jump kernels and processes with jump kernels comparable to that of subordinated diffusion process.
\begin{theorem}[Characterization of jump kernels] \label{t:char}
	Let $(M,d,\mu)$ a unbounded, complete, separable, locally compact metric measure space that satisfies \hyperlink{vd}{$\operatorname{VD}$}. 
	Let $(\sE^c,\sF^c)$ be a strongly local, regular Dirichlet form on $L^2(M,\mu)$ that satisfies \hyperlink{phi}{$\on{PHI(\psi_c)}$}, where $\psi_c$ is a scale function. Let $X$ be the $\mu$-symmetric Hunt process corresponding to $(\sE^c,\sF^c)$ on $L^2(M,\mu)$. Given a scale function $\psi_j$, the following are equivalent.
	\begin{enumerate}[(a)]
		\item There exists a  regular Dirichlet  form $(\sE^j,\sF^j)$ on $L^2(M,\mu)$ of pure jump type  whose jump kernel satisfies \hyperlink{jpsij}{$\on{J(\psi_j)}$}.
		\item There exists a subordinator $S$ such that the subordinated process $X \circ S$ corresponds to a regular Dirichlet form $(\sE^j,\sF^j)$ on $L^2(M,\mu)$  of pure jump type and  satisfies  \hyperlink{jpsij}{$\on{J(\psi_j)}$}.
		\item The scale function $\psi_j$ satisfies
		\be \label{e:crit}
		\int_0^1 \frac{\psi_c(s)}{ s \psi_j(s)} \, ds < \infty.
		\ee
	\end{enumerate}
\end{theorem}
\begin{remark} \label{r:prev} {\rm 
	By \cite[Corollary 1.10]{Mur}, we have $\psi_c(s) \lesssim s^2$ for all $s \in [0,1]$. Therefore the condition
	\be \label{e:suff}
	\int_0^1 \frac{s}{\psi_j(s)} \,ds < \infty
	\ee
	implies \eqref{e:crit}.
The above sufficient condition \eqref{e:suff} for \eqref{e:crit} was assumed in the context of jump processes on $d$-regular sets in the Euclidean space \cite[eq. (1.3)]{CK}. Furthermore, the Brownian motion on Euclidean space satisfies \hyperlink{phi}{$\on{PHI(\psi_c)}$} with $\psi_c(r)=r^2$ in which case \eqref{e:crit} is same as \eqref{e:suff}. The integrablity condition \eqref{e:crit} was recently used to obtain heat kernel estimates in \cite[eq. (2.19)]{BKKL}. Theorem \ref{t:char} can be viewed as a justification for the assumptions in these previous works. Although these versions of \eqref{e:suff} were used in earlier works to obtain heat kernel bounds and parabolic Harnack inequality, the necessity of \eqref{e:crit} is new and is the key contribution of our work. 
}
\end{remark}

\begin{cor}[Parabolic Harnack inequality via subordination] \label{c:main}
		Let $(M,d,\mu)$ a unbounded, complete, separable, locally compact metric measure space that satisfies \hyperlink{vd}{$\operatorname{VD}$}. 
	Let $(\sE^c,\sF^c)$ be a regular, strongly local Dirichlet form on $L^2(M,\mu)$ that satisfies \hyperlink{phi}{$\on{PHI(\psi_c)}$}, where $\psi_c$ is a scale function. Let $(\sE^j,\sF^j)$ be a regular Dirichlet form  $L^2(M,\mu)$ of pure jump type that satisfies \ref{jpsij} for some scale function $\psi_j$.
	
	\begin{enumerate}[(a)]
		\item 
 Then  jump type Dirichlet form $(\sE^j,\sF^j)$ satisfies \hyperlink{phi}{$\on{PHI(\wh \psi_j)}$}, where $\wh \psi_j$ is a scale function satisfying the following estimate: there exists $C \ge  1$ such that
	\[
	 C^{-1}\frac{\psi_c(r)}{\int_0^r   \frac{\psi_c(s)}{ s \psi_j(s)} \, ds }  \le  \wh \psi_j(r) \le C \frac{\psi_c(r)}{\int_0^r   \frac{\psi_c(s)}{ s \psi_j(s)} \, ds } \q \mbox{for all $r>0$.}
	\]	
\item The scale functions $\psi_c,\psi_j, \wh \psi_j$ satisfy the following estimates:
\begin{align} \label{e:cor1}
\wh \psi_j(r) &\lesssim \psi_j(r) \q \mbox{for all $r >0$,}\\
\psi_c(r) & \lesssim \psi_j(r) \q \mbox{for all $r \le 1$,} \label{e:cor2}\\
\frac{\wh \psi_j(R)}{\wh \psi_j(r)} &\lesssim \frac{\psi_c(R)}{\psi_c(r)} \q \mbox{for all $0< r \le  R$,} \label{e:cor3}\\
\wh \psi_j(r) \lesssim \psi_c(r) \q \mbox{for all $r \ge 1$,} &\mbox{  and }
\psi_c(r) \lesssim   \wh \psi_j(r)\q \mbox{for all $r \le 1$.} \label{e:cor4}
\end{align}
		\end{enumerate}
\end{cor}
We provide a probabilistic interpretation of the Corollary \ref{c:main}(b).
Let $(X_t)_{t \ge 0}, (Y_t)_{t \ge 0}$ denote the diffusion and jump processes corresponding to the Dirichlet forms $(\sE^c,\sF^c)$ and $(\sE^j,\sF^j)$ in Corollary \ref{c:main} respectively. Then by the results of \cite{GHL15,GT12,CKW2}, the function  $\psi_c$ and $\wh \psi_j$  govern the exit times from balls of the processes $X$ and $Y$ respectively.
In particular, the following two sided bounds for exit times hold: 
\[
\bE_x [\tau^X_{B(x,r)}] \asymp  \psi_c(r),\qq \bE_x [\tau^Y_{B(x,r)}] \asymp  \wh \psi_j(r)
\]
for all $x \in M, r>0$, where $\bE_x$ denote the expectation when the process starts at $x$ and $\tau^X_{B(x,r)},\tau^Y_{B(x,r)}$ corresponds to the exit time of $B(x,r)$ for the process $X$ and $Y$ respectively. 
By the last estimate in Corollary \ref{c:main}(b), the diffusion process exits smaller balls (say balls of radii less than $1$) faster than the jump process. On the other hand, the jump process exits larger balls faster than the diffusion process.  A similar assumption can be found in \cite[(1.13)]{CKW3}. This work grew from an attempt to understand and justify the above mentioned assumptions in \cite{CKW3, BKKL, CK}.
\section{Subordinator with comparable jump kernel}
We recall the following result from \cite{BKKL}.
We emphasize that  the following result of does not require $\psi_c$ and $\psi_j$ to satisfy condition \eqref{e:crit}. 
In the notation of  \cite{BKKL}, we only need that $\mathcal{J}_\psi$ is defined by the equation that is four lines above the statement of \cite{BKKL}.
\begin{lem} \cite[Lemma 4.2]{BKKL} \label{l:levyjump}
	Let $\psi_c, \psi_j$ be scale functions and
let $(t,x,y) \mapsto p^c_t(x,y)$ be a heat kernel that satisfies the estimate \hyperlink{hke}{$\on{HKE}(\psi_c)$}. Then
	\be
	\label{e:jump}
	\int_0^\infty \frac{p^c_t(x,y)}{t \psi_j \circ \psi_c^{-1}(t)}\,dt  \asymp \frac{1}{V(x,d(x,y)) \psi_j(d(x,y))} \qq \mbox{for all $x,y \in M$.}
	\ee
\end{lem}
 The following Poincar\'e inequality follows from the parabolic Harnack inequality and is a crucial ingredient in our proof.
\begin{lem} [Poincar\'e inequality]\label{l:poin} \cite[Theorem 1.2]{GHL15}
Let $(M,d,\mu)$ an unbounded, complete, separable metric measure space. 
Let $(\sE^c,\sF^c)$ be a regular, strongly local Dirichlet form on $L^2(M,\mu)$ that satisfies \hyperlink{phi}{$\on{PHI(\psi_c)}$}, where $\psi_c$ is a scale function. Then we have the following Poincar\'e inequality: there exist $C_P, A>1$ such that for any ball $B(x,r)$ and for any function $f \in \sF^c$, 
\be \label{poin} \tag{$\on{PI(\psi_c)}$}
\int_{B(x,r)} (f(y)-f_{B(x,r)})^2 \, \mu(dy) \le C_P \int_{B(x,Ar)} \int_{B(x,Ar)} d \Gamma(f,f),
\ee
where $\Gam(f,f)$ denotes the energy measure, and $f_{B(x,r)}$ denotes the $\mu$-average of $f$ in $B(x,r)$ defined by $f_{B(x,r)} = \frac{1}{\mu(B(x,r))} \int_{B(x,r)} f \, d \mu$.
\end{lem}
\begin{proof}
	This is an immediate consequence of \cite[Theorem 1.2]{GHL15} and Theorem \ref{t:phi-hke}.
\end{proof}

The following lemma is classical and is a special case of \cite[Theorem 2.1]{Oku}
\begin{lem} \cite[Theorem 2.1]{Oku} \label{l:subjump} Let $X$ be a $\mu$-symmetric process with a conservative semigroup whose heat kernel is $p_t^c(\cdot,\cdot)$, and $(S_t)$ be a subordinator with L\'evy measure $\nu$. Then the Dirichlet form corresponding to the $\mu$-symmetric subordinated process $Y_t= X_{S_t}$ is a pure jump process with jump kernel
	\[
	J(x,y) = \frac 1 2 \int_0^\infty p^c_t(x,y) \,\nu(dt).
	\]
\end{lem}

The following elementary estimate along with Lemma \ref{l:subjump} provides the desired bounds on the jump kernel of subordinated process.
%
%

\begin{prop} \label{p:levymeas}
	Let $(M,d,\mu)$ an unbounded, complete, separable metric measure space. 
	Let $(\sE^c,\sF^c)$ be a regular, strongly local Dirichlet form on $L^2(M,\mu)$ that satisfies \hyperlink{phi}{$\on{PHI(\psi_c)}$}, where $\psi_c$ is a scale function.  Let  $(\sE^j,\sF^j)$ be a pure jump type Dirichlet form  such that the corresponding jumping measure satisfies  \hyperlink{jpsijge}{$\on{J(\psi_j)_\ge}$}, where $\psi_j$ is a scale function. Then 
	\[
	\int_0^1 \frac{1}{ \psi_j \circ \psi_c^{-1}(t)} \, dt <\infty.
	\]
\end{prop}
\proof
Assume to the contrary that
\be \label{e:lc1}
	\int_0^1 \frac{1}{ \psi_j \circ \psi_c^{-1}(t)} \, dt =\infty.
\ee
Since $t\mapsto \psi_j \circ \psi_c^{-1}(t)$ is a continuous positive function on $(0,\infty)$, we have 
\be \label{e:lc1'}
\int_0^{t_0} \frac{1}{ \psi_j \circ \psi_c^{-1}(t)} \, dt =\infty,
\ee
for any $t_0>0$.
 Let $(P_t^c)_{t>0}$ denote the Markov semigroup corresponding to the Dirichlet form $(\sE^c,\sF^c)$ on $L^2(M,\mu)$ and let $p_t^c(\cdot,\cdot)$ denote the corresponding heat kernel (which exists by Theorem \ref{t:phi-hke}). Note that
 \[
 \frac{1}{t}\langle f-P_t^c f, f\rangle = \frac{1}{2t} \int_M \int_M p_t^c(x,y) (f(x)-f(y))^2 \, \mu(dx)\,\mu(dy).
 \]
Therefore by Lemma \ref{l:levyjump}, there exists $C_1>0$ such that 
\be
\label{e:lc2}
C_1^{-1} \sE^j(f,f) \le \int_{0}^\infty \frac{1}{t \psi_j \circ \psi_c^{-1}(t)}\langle f-P_t^c f, f\rangle \,dt \le C_1 \sE^j(f,f)\q \mbox{for all $f \in \sF^j$.}
\ee

Let $f \in \sF^j$. Choose $N,\eps \in  (0,\infty)$.
We define $\sE^c(f,f)$ for any $f \in L^2(M,\mu)$ by setting $\sE^c(f,f)=\infty$ whenever $f \in L^2(M,\mu) \setminus \sF^c$.
 Let $(P_t^c)_{t>0}$ denote the Markov semigroup corresponding to the Dirichlet form $(\sE^c,\sF^c)$ on $L^2(M,\mu)$.
By \eqref{e:df3}, there exists $t_0 \in (0,1)$ (depending on $f,N,\eps$) such that for any $t \in (0,t_0)$,
\be \label{e:lc3}
\frac{1}{t}\langle f-P_t^c f, f\rangle \ge (N \wedge \sE^c(f,f)) - \eps.
\ee
Combining \eqref{e:lc2} and \eqref{e:lc3}, 
\begin{align}
\infty &> \sE^j(f,f) \\
& \ge C_1^{-1}\int_0^{t_0} \frac{1}{t \psi_j \circ \psi_c^{-1}(t) }\langle f-P_t^c f, f\rangle \,dt \\
& \ge \left[(N \wedge \sE^c(f,f)) - \eps\right]  \int_0^{t_0}\frac{1}{ \psi_j \circ \psi_c^{-1}(t)} \, dt \label{e:lc4}
\end{align}
Combining \eqref{e:lc1'} and \eqref{e:lc4}, we obtain that $N \wedge \sE^c(f,f)\le \eps$. By letting   $\eps \to 0$, and using \eqref{e:df1}, \eqref{e:df2}, we obtain that 
\be \label{e:lc5}
\sE^c(f,f) =0, \q \mbox{for all $f \in \sF^j$.}
\ee 
By the Poincar\'e inequality (Lemma \ref{l:poin}), this implies that any $f \in \sF^j$ is constant $\mu$-almost everywhere on every ball $B(x,r)$. In particular, every function in $\sF^j$ is constant $\mu$-almost everywhere. This implies that $\sF^j$ is not dense in $L^2(M,\mu)$, contradicting the assumption that $(\sE^j,\sF^j)$ is a Dirichlet form.
\qed

The following result is an elementary consequence of change of variables formula.
\begin{lem} \label{l:integral}
	Let $\psi_c,\psi_j$ be scale functions. Then $	\int_0^1 \frac{\psi_c(s)}{s\psi_j(s)} \,ds < \infty$ is equivalent to $	\int_0^1 \frac{1}{ \psi_j \circ \psi_c^{-1}(t)} \, dt <\infty$.
\end{lem}
\proof 
It is easy to verify using \eqref{e:sf} that $\psi_c(t)$ is comparable to the function
\[
t \mapsto \int_0^t \frac{\psi_c(r)}{r} \,dr.
\]
Therefore, we assume without loss of generality that $\psi_c$ is continuously differentiable and
\begin{equation} \label{e:cv1}
\psi_c'(r) \asymp \frac{ \psi_c(r)}{r}
\end{equation}
for all $r>0$. Using \eqref{e:cv1} and substituting $t=\psi_c(s)$ in the integral $\int_0^1 \frac{1}{ \psi_j \circ \psi_c^{-1}(t)} \, dt$, we obtain the desired equivalence.
\qed
\subsection{Proof of the main results}
\noindent {\em Proof of Theorem \ref{t:char}.} 
The parabolic Harnack inequality implies that $(M,d)$ is connected \cite[Proposition 5.6]{GH}, \cite[Theorem 5.4]{BCM}. By \cite[Exercise 13.1]{Hei}, $(M,d)$ satisfies \hyperlink{rvd}{$\operatorname{RVD}$}.

That (b) implies (a) is obvious. 

Next, we show that (c) implies (b).
By Lemma \ref{l:integral}, the measure $\nu(t):= \frac{1}{t \psi_j \circ \psi_c^{-1}(t)} \,dt$ is a L\'evy measure of subordinator. Let $S$ denote the subordinator corresponding to the L\'evy measure $\nu$. By \cite[Theorems 1.2 and 1.3]{GHL15} and Theorem \ref{t:phi-hke}, the Markov semigroup corresponding to $(\sE^c,\sF^c)$ is conservative. By Lemmas \ref{l:subjump} and \ref{l:levyjump}, the subordinated diffusion process $X \circ S$ is a $\mu$-symmetric pure jump process whose jump \hyperlink{jpsij}{$\on{J(\psi_j)}$}.

Finally,  (a) implies (c), following from Proposition \ref{p:levymeas} and Lemma \ref{l:integral}.
\qed

\noindent {\em Proof of Corollary \ref{c:main}.} 
\begin{enumerate}[(a)]
	\item By Theorem \ref{t:char}, we obtain \eqref{e:crit}. Therefore, by \cite[Theorem 2.19 and Lemma 4.5]{BKKL},  Lemma \ref{l:integral},  along with the characterization of parabolic Harnack inequality in \cite[Theorem 1.20]{CKW2}, we obtain (a).
	\item 
	The  estimate \eqref{e:cor1} follows from 
	\[
	\int_0^r   \frac{\psi_c(s)}{ s \psi_j(s)} \, ds  \ge \int_{r/2}^r   \frac{\psi_c(s)}{ s \psi_j(s)} \, ds  \ge   \frac{\psi_c(r/2)}{2\psi_j(r)} \gtrsim\frac{\psi_c(r)}{\psi_j(r)}.
	\]
	Using the above estimate 
	\[
	\frac{\psi_c(r)}{\psi_j(r)} 
	\lesssim \int_{r/2}^r   \frac{\psi_c(s)}{ s \psi_j(s)} \, ds \le \int_{0}^1   \frac{\psi_c(s)}{ s \psi_j(s)} \, ds
	\]
	for any $r \le 1$. This along with \eqref{e:crit} in Theorem \ref{t:char} yields \eqref{e:cor2}.
	By part (a), for any $0 < r \le R$, we obtain
	\[
	\frac{\wh \psi_j(R)}{\wh \psi_j(r)} \frac{\psi_c(r)}{\psi_c(R)} \lesssim  \frac {\int_0^r   \frac{\psi_c(s)}{ s \psi_j(s)} \, ds } {\int_0^R   \frac{\psi_c(s)}{ s \psi_j(s)} \, ds } \lesssim 1.
	\]
	This implies \eqref{e:cor3}. The estimates in \eqref{e:cor4} follow from \eqref{e:cor3} by substituting $r=1$ and $R=1$ respectively.
\end{enumerate}
\qed

\subsection*{Acknowledgements} 
We thank Professors Zhen-Qing Chen and  Takashi Kumagai for  helpful remarks and a reference. In particular, we learned about \cite{Oku} from ZQC.

\noindent  \\
Department of Mathematical Sciences, Tsinghua University, Beijing 100084, China.\\
liu-gh17@mails.tsinghua.edu.cn \sms\\

\noindent Department of Mathematics, University of British Columbia,
Vancouver, BC V6T 1Z2, Canada. \\
mathav@math.ubc.ca

\end{document}